%% file: graphs_on_nc.tex
\documentclass[oneside,9pt]{article} \topmargin -15mm \textwidth 160 true mm
\usepackage{t1enc}
\usepackage{graphicx}
\usepackage{float}
\usepackage{caption}
\usepackage{mdframed,lipsum}
\usepackage{mathtools} 
\usepackage{subcaption}

\usepackage{epstopdf}
\usepackage{mathrsfs}
\usepackage{multicol}
\usepackage{setspace}
\usepackage{amsmath, amssymb, amsfonts, amstext, amsthm,amscd}
\usepackage{relsize}
\usepackage{stmaryrd}
\usepackage{mathtools}
\DeclareSymbolFont{symbolsC}{U}{txsyc}{m}{n}
\DeclareMathSymbol{\Searrow}{\mathrel}{symbolsC}{117}
\usepackage[mathscr]{euscript}
\usepackage{enumerate, mathtools}
\usepackage{chngcntr}
\usepackage{tikz-cd}
\usepackage{tikz,color,soul, verbatim}
\usetikzlibrary{patterns}
\usepackage{wrapfig}
\usepackage{subcaption}
\usepackage{soul}
\usetikzlibrary{arrows.meta,arrows,shapes.geometric,mindmap,shapes}
\usepackage[english]{babel}
\usepackage[pdftex]{hyperref}
\newtheorem{theorem}{Theorem}[section]
\newtheorem{proposition}[theorem]{Proposition}
\newtheorem{lemma}[theorem]{Lemma}

\newtheorem{corollary}[theorem]{Corollary}
\theoremstyle{definition}
\newtheorem{definition}[theorem]{Definition}
\newtheorem{eg}[theorem]{Example}
\newtheorem{remark}[theorem]{Remark}

\newtagform{clai}{section}{Claim}

\numberwithin{equation}{section}

\numberwithin{mytheorem}{subsection}
\def \no {\noindent}
\def \R{\mathbb{R}}

\def \N{\mathbb{N}}
\def \C {\mathcal{C}}
\def \U {\mathcal{U}}
\def \rar{\rightarrow}

\def \D{\mathcal{D}}

\newcommand\rddots{\ooalign{\scalebox{-1}[1]{$\ddots$}}}

\newcommand{\ring}[1]{\mathcal{R}_{#1}}

\newcounter{casenum}

\tikzstyle{vertex}=[circle,draw, inner sep=0pt, minimum size=4pt]
\newcommand{\vertex}{\node[vertex]}
\newcommand{\CF}[1]{\operatorname{CF}\left(\mathcal{J}_{#1}\right)}
\newcommand{\Tk}[2]{{\operatorname{Tk}_{#1}}(#2)}
\newcommand{\CC}[1]{{\mathcal{C}^{\mathscr{C}}_{#1}}}
\newcommand{\CR}[1]{{\mathcal{C}^{\mathscr{R}}_{#1}}}
\newcommand{\ideal}[1]{\langle{#1}\rangle}
\newcommand{\J}{\mathcal{J}}
\newcommand{\nei}[1]{\mathcal{J}_{#1}}
\newcounter{subcasenum}

\title{Properties of graphs of neural codes}
\author{Suhith K N\thanks{Suhith's research is partially supported by Inspire fellowship from DST grant IF190980.} \ and Neha Gupta}
\date{}
\begin{document}

	\maketitle
	\begin{abstract}
		A neural code on $ n $ neurons is a collection of subsets of the set $ [n]=\{1,2,\dots,n\} $. In this paper, we study some properties of graphs of neural codes. In particular, we study codeword containment graph (CCG) given by Chan et al. (SIAM J. on Dis. Math., 37(1):114-145,2017) and general relationship graph (GRG) given by Gross et al. (Adv. in App. Math., 95:65-95, 2018). We provide a sufficient condition for CCG to be connected. We also show that the connectedness and completeness of CCG is preserved under surjective morphisms between neural codes defined by A. Jeffs (SIAM J. on  App. Alg. and Geo., 4(1):99-122,2020). Further, we show that if CCG of any neural code $\C$ is complete with $|\C|=m$ then $\C \cong \{\emptyset,1,12,\dots,123\cdots m\}$ as neural codes. We also prove that a code whose CCG is complete is  open convex. Later, we show that if a code $\C$ with $|\C|>3$ has its CCG to be connected 2-regular then $|\C| $ is even. The GRG was defined only for degree two neural codes using the canonical forms of its neural ideal. We first define GRG for any neural code. Then, we show the behaviour of GRGs under the various elementary code maps.  At last, we compare these two graphs for certain classes of codes and see their properties.
		
	\end{abstract}
	
	
	\input{main}
	\bibliographystyle{plain}
	\bibliography{refs}
	%\bibliography{refs}
	%	\input{Appendex1}
	Authors Affiliation:  \vspace{0.2cm}\\
	Suhith K N \\ Research Scholar \\ Department of Mathematics \\ Shiv Nadar Institution of Eminence (Deemed to be University) \\ Delhi-NCR \\ India \\ Email: sk806@snu.edu.in \\
	Suhith's research is partially supported by Inspire fellowship from DST grant IF190980. \vspace{0.2cm} \\
	Neha Gupta \\ Assistant Professor \\ Department of Mathematics \\ Shiv Nadar Institution of Eminence (Deemed to be University) \\ Delhi-NCR \\ India \\ Email: neha.gupta@snu.edu.in  \vspace{0.6cm}\\
	
\end{document}

%% file: main.tex
\section{Introduction}
A neural code on $ n $ neurons is a collection of subsets of the set $ [n]=\{1,2,\dots,n\}. $ Given a collection of sets $ \U=\{U_1,U_2,\dots,U_n\} $ in $ \R^d, $ we associate a neural code $ \C(\U)= \{\sigma\subseteq [n]\mid \cap_{j\in\sigma}U_j \backslash \cup_{i\notin \sigma}U_i \not = \emptyset\}. $ A neural code $ \C $ is said to be realizable if there exists a collection of sets $ \U $ in $ \R^d $ such that $ \C=\C(\U). $ Additionally, we refer to $ \C $ as (open) convex realizable if $ U_i $'s are all (open) convex. Similarly, there is a notion of closed convex codes. Franke and Muthiah \cite{franke2018every} showed that every code is convex realizable. However, there are neural codes that are not open convex realizable and not closed convex realizable. 
 
 In these constrained circumstances, it is challenging to determine when a neural code is open convex realizable. Algebraic techniques, on the other hand, have been successful to gain insight into obstructions to codes being realizable. For a given code $ \C, $ Curto et al. \cite{curto2013neural} defined the associated neural ring, $ \ring{\C} $ and the neural ideal, $ \mathcal{J}_\C \subseteq \mathbb{F}_2[x_1,x_2,\dots,x_n] $. These algebraic have been defined to study the properties of a neural code. In order to identify whether the codes are open convex, Curto et al. \cite{curto2019algebraic} worked with $\J_\C$ and its standard generating set, the canonical form $\CF{\C}$. They provide conditions on canonical form of the code for certain families of convex codes including the class of \textit{intersection-complete codes} given by Cruz et al. \cite{cruz2019open}. 
 
 The other interesting technique used to study the open convexity of codes is to study an underlying graph of the code. Gross et al. \cite{gross2018neural} defined \textit{general relationship graph} for any neural code $\C$. This graph is solely dependent on canonical form, $\CF{\C}$ of the neural ideal $\J_\C$. Gross et al. \cite{gross2018neural} also defined  inductively pierced codes and gave a relation with general relationship graphs.    
  Curry et al. \cite{youngs2023recognizing} extended the work on inductively pierced codes. Initially they categorize these neural codes with respect to their general relationship graph. Later, they provide an open convex realization for these codes in Euclidean space $\R^d,$ for some $d>0$. This motivates us to further study the properties of general relationship graph.   On the other hand, \cite{chan2023nondegenerate} worked on closed convex codes. They define a graph called \textit{codeword containment graph}. They show that if this graph satisfies certain conditions then the the code is not closed convex.  
  
  In this paper, we will work with codeword containment graph (CCG) and general relationship graph (GRG), and explore their properties. The paper is structured as follows: In Section \ref{sec:pre}, we will introduce certain preliminaries required for the paper. In Section \ref{sec:ccg}, we will study CCG.  We further provide a sufficient condition for CCG to be connected. We also show that the connectedness and completeness of a CCG is preserved under surjective morphisms between neural codes \cite{jeffs2020morphisms}. Further, we show that if CCG of $\C$ is complete with $|\C|=m$ then $\C \cong \{\emptyset,1,12,\dots,123\cdots m\}$ as neural codes. We also prove that a code whose CCG is complete is  open convex. Later, we show that if a code $\C$ with $|\C|>3$ has its CCG as connected 2-regular then $|\C| $ is even. Moreover, for any $k>2$ we show that $  \{1,2,\dots,k,12,23,\dots,k-1k,1k\}$ is a code with with cardinality $2k$ and its CCG is connected 2-regular. In section \ref{sec:cf}, we will study canonical forms of pseudo-monomial ideals and define neural ideals for a given neural code. We further use the algorithm provided by Youngs \cite{youngs2014neural} to obtain canonical forms. With these findings we discuss how canonical forms behave under various elementary code maps \cite{curto2020neural}.  Note that the GRG is defined only for degree two neural codes \cite{gross2018neural} using the canonical forms of its neural ideal. In Section \ref{sec: grg},  we first define a general relationship complex for any given neural code and then use its 1-skeleton as GRG for any neural code. Later, we will we show the behaviour of GRGs under the elementary code maps using the results obtained in Section \ref{sec:cf}. In the last section, we compare these two graphs for two classes of codes called complete and 2-regular codes.

\section{Preliminaries} \label{sec:pre}
We will refer to neural codes as simply codes throughout the rest of the paper. In this section, we will understand two kinds of maps between codes. The first kind is the elementary code maps given by Curto and Youngs \cite{curto2020neural}, and the second kind of maps are given by Jeffs\cite{jeffs2020morphisms}, called as morphisms. 
\subsection{Neural ring homomorphisms} \label{nrh}
Curto et al. \cite{curto2013neural} associated a ring  $ \ring{\C} $ for a given code $ \C $, and  called it the neural ring associated to the code $\C$. For a given code $ \C, $ the neural ring, $\ring{\C}$ is given by, $ \ring{\C}=\mathbb{F}_2[y_1,\dots,y_n]/I_{\C}, $ where $ I_\C=\{f\in\mathbb{F}_2[x_1,x_2,\dots,x_n] | f(c)=0 \text{ for all } c\in \C\}. $  Further, Curto and Youngs \cite{curto2020neural}, studied the ring homomorphisms between two neural rings. They showed that there is a 1-1 correspondence between code maps\footnote{A code map $q:\C\to\D$ is a function that takes codewords of $\C$ to codewords of $\D$.} $ q:\C\rar \D $ and the ring homomorphisms $ \phi:\ring{\D}\rar \ring{\C}. $ The associated code map of the ring homomorphism $ \phi $ is denoted as $ q_\phi $.  They also demonstrated that $ |\C|=|\D| $ is the only condition under which $ \ring{\C} \cong \ring{\D} $. This means that the ring homomorphisms ignore the nature of codewords that are present in the code and only take its cardinality into account. So, to capture the essence of the codewords of a code and not just its cardinality, \cite{curto2020neural} added some additional pertinent conditions to the ring homomorphisms and named these maps as neural ring homomorphisms. We will now give detailed definition of neural ring homomorphisms.
\begin{definition}[Neural ring homomorphism]\cite[Definition 3.1]{curto2020neural}
	Let $ \C  $ and $ \D $ be codes on $ n $ and $ m $ neurons respectively, and let $ \ring{\C}= \mathbb{F}_2[y_1,\dots,y_n]/I_{\C} $ and $ \ring{\D}= \mathbb{F}_2[x_1,\dots,x_m]/I_{\D}  $ be the corresponding neural rings.  A ring homomorphism $ \phi:\ring{\D}\rar \ring{\C} $ is a neural ring homomorphism if $ \phi(x_j)\in\{y_i\mid i \in [n]\} \cup \{0,1\} $ for all $ j\in [m].$  We say that a neural ring homomorphism $ \phi $ is a neural ring isomorphism if it is a ring isomorphism and its inverse is also a neural ring homomorphism. 
\end{definition}    
\no Further, Curto and Youngs showed that given any neural ring homomorphism $ \phi $, the associated code map $ q_\phi $ can only be a composition of five elementary code maps. We state this theorem below.
\begin{theorem}\cite[Theorem 3.4]{curto2020neural} \label{thmnycc1}
	A map $ \phi:\ring{\D}\rar \ring{\C} $ is a neural ring homomorphism if and only if $ q_\phi $ is a composition of the following code maps: \\ 1. Permutation, 2. Adding a trivial neuron (or deleting a trivial neuron), 3. Duplication of a neuron (or deleting a neuron that is a duplicate of another), 4. Neuron projection (or deleting a not necessarily trivial neuron), 5. Inclusion (of one code into another).
	\\ \no 	Moreover, $ \phi $ is a neural ring isomorphism if and only if $ q_\phi $ is a composition of maps $ (1)-(3). $   
\end{theorem}
\no These five types of code maps are referred as elementary code maps. We will work with these maps in section \ref{sec: grg}. 
\subsection{Morphisms}
In this part of the section, we will give a detailed description of morphisms introduced by Jeffs \cite{jeffs2020morphisms}. 
\begin{definition} \cite[Definition 1.1]{jeffs2020morphisms}
	Let $\C$ be a code on $n$ neurons, and let $\sigma\subseteq[n]$. The \emph{trunk} of $\sigma$ in $\C$ is defined to be the set $$\operatorname{Tk}_\C(\sigma)= \{c\in\C \mid \sigma \subseteq c\}.$$   
\end{definition}
\no A subset of $\C$  is called a trunk in $\C$  if it is empty, or equal to $\operatorname{Tk}_\C(\sigma)$ for some $\sigma  \subseteq  [n]$.
\begin{definition} \cite[Definition 1.2]{jeffs2020morphisms} \label{jeffmorphis}
	Let $\C$  and $\D$  be codes. A function $f : \C  \rightarrow  \D $ is a \emph{morphism} if for every
	trunk $T \subseteq  \D$  the pre-image $f^{-1} (T )$ is a trunk in $\C$. A morphism is an  \emph{isomorphism} if it has an inverse function which is also a morphism. 
\end{definition}
If there exists a code map $q:\C\to\D$ which is an isomorphism then we say that $\C$ and $\D$ are isomorphic to each other. Further, observe that the elementary code maps of Theorem \ref{thmnycc1} also satisfies the above definition. So, given any neural ring homomorphism $\phi:\ring{\D} \to \ring{\C}$, the associated code map $q_\phi: \C\to \D$ is a morphism. Trunks of single neurons play a vital role and Jeffs calls them simple trunks. 
\begin{definition}\cite[Definition 2.4]{jeffs2020morphisms}
	For any $i\in[n]$ trunks of the form $\operatorname{Tk}(\{i\})$  will be called \emph{simple trunks} and be denoted by
	$\operatorname{Tk}(i)$.
\end{definition} 
In the next proposition Jeffs\cite{jeffs2020morphisms} explains that it is enough to study the inverse images of simple trunks to see if a map is a morphisms, i.e.,
\begin{proposition}\cite[Proposition 2.5]{jeffs2020morphisms} \label{propsimtrun}
	Let $\C  \subseteq  2^{[n]}$ and $\D  \subseteq  2^{[m]}$ be codes. A function $f : \C  \rightarrow  \D$  is a morphism
	if and only if for every $i \in  [m]$, $f^{- 1} (\Tk{\D}{i})$ is a trunk in $\C$.
\end{proposition}

Further Jeffs \cite{jeffs2020morphisms} has studied various properties of morphisms. Now, we mention one such property that we will use in Section \ref{sec:ccg}. 
\begin{proposition}\cite[Proposition 2.6]{jeffs2020morphisms}
Morphisms are monotone: if $f : \C  \rightarrow  \D$  is a morphism and $ c_1 , c_2 \in  \C $ are
	such that $c_1 \subseteq  c_2$ , then $f (c_1 )\subseteq  f (c_2 ).$ \label{proptrunmomon}
\end{proposition}
\section{Codeword containment graph (CCG)} \label{sec:ccg}
In this section we first see the definition of codeword containment graph. This graph was first used by Chan et al. \cite{chan2023nondegenerate} to study closed convex codes. We will provide a sufficient condition for this graph to be connected. Further, we will show that the morphisms will preserve the connectedness property of this graph. 
\begin{definition}[Codeword containment graph (CCG)]\cite{chan2023nondegenerate} 
	The codeword containment graph of a code $ \C $ is the (undirected) graph with vertex set consisting of all codewords of $ \C $ and edge set $ \{(\sigma,\tau)\mid \sigma\subsetneq \tau \text{ or } \tau\subsetneq \sigma\}. $ Denote $ G_\C $  to be the codeword containment graph of the code $ \C. $
\end{definition}
\begin{eg}
We show the CCG for codes $\{1,2,13,123\}$ and $\{13,125,1235,1245\}.$ 
\begin{figure}[h]
	\centering
\begin{subfigure}[b]{0.5\linewidth}
	\centering
		\begin{tikzpicture}  
		\vertex (1) at (0,0) [label=below:1] {};
		\vertex (2) at (1,0) [label=below:2]{};
		\vertex (13) at (0,1) [label=left:13]{};
		\vertex (123) at (1,1.5) [label=right:123]{};
		\path[-]
		(1) edge (13)
		(2) edge (123)
		(1) edge (123)
		(13) edge (123)
		;
	\end{tikzpicture}
	\caption{$G_\C$ when $\C=\{1,2,13,123\}$}
\end{subfigure}
	\begin{subfigure}[b]{0.2\linewidth}
		\centering 
		\begin{tikzpicture}  
			\vertex (13) at (0,0) [label=below:13] {};
			\vertex (125) at (1,0) [label=below:125]{};
			\vertex (1235) at (0,1) [label=left:1235]{};
			\vertex (1245) at (1,1) [label=right:1245]{};
			\path[-]
			
			(13) edge (1235)
			(125) edge (1235)
			
			(125) edge (1245)
			;
		\end{tikzpicture}
		\caption{$G_\C$ when  $\C=\{13,125,1235,1245\}$}
		\label{ncccg}
	\end{subfigure}
	\caption{Codeword containment graphs}
	\label{ccg}
\end{figure}
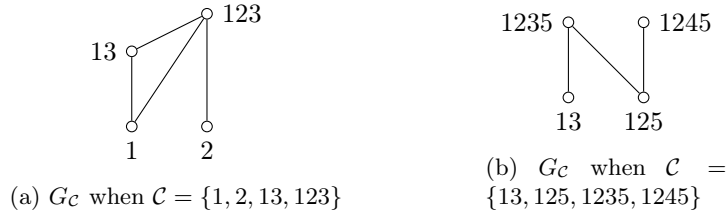
\end{eg}
\subsection{Connectedness of CCG}
In this part of the section we will discuss the connectedness of CCG. We will start with a trivial remark. 
\begin{remark}
	If $ \emptyset\in\C $ then $ G_\C $ is connected. This happens because the vertex $\emptyset$ has an edge with any vertex $\sigma$ of $G_\C$. Therefore there is always a path between any two vertices $\sigma$ and $\tau $ of $G_\C$ via the vertex $\emptyset$. For example look at Figure \ref{conempt}.
	\begin{figure}[h!]
		\centering
		\begin{tikzpicture}  
			\vertex (1) at (0,1) [label=above:1] {};
			\vertex (2) at (1,0) [label=below:2]{};
			\vertex (4) at (-2,-.5) [label=left:4]{};
			\vertex (0) at (-1,-0.5) [label=below:$\emptyset$]{};
			\vertex (123) at (1,1.5) [label=right:123]{};
			\path[-]
			
			(2) edge (123)
			(1) edge (123)
			(0) edge (4)
			(1) edge (0)
			(2) edge (0)
			(0) edge (123)
			;
		\end{tikzpicture}
		\caption{$G_\C$ when $\C=\{\emptyset,1,2,123,4\}$}
		\label{conempt}
	\end{figure}
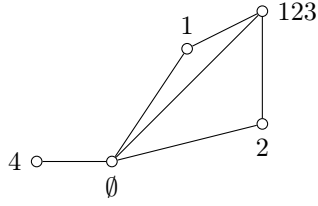
	
\end{remark}
Note that,  however CCG can still be connected without $\emptyset\in \C$. So we now look at a different sufficient condition for a CCG to be connected. Before we understand the same, we will learn about simplicial complex of a code. 

\begin{definition}[Simplicial complex of a code] Given a code $ \C $ on $ n $ neurons we denote the set $ \Delta(\C) $ as the simplicial complex of the code and it is given by $$\Delta(\C)=\{\alpha\subseteq[n]\mid\alpha\subseteq\beta, \text{ for some } \beta\in\C\}.$$ Note that by the above definition, we always have $ \C\subseteq \Delta(\C). $ \label{defscofc}
\end{definition} 
\no The next result states the sufficient condition for the codeword containment graph to be connected.
\begin{proposition}
	The codeword containment graph 	$ G_\C$ is connected if for every pair $ \sigma,\tau\in\C $, $ \sigma\cup\tau\in\Delta(\C). $ 
\end{proposition} 
\begin{proof}
	Consider $ \sigma,\tau\in\C $.	Let $ \sigma\cup\tau\in\Delta(\C) $, then there exists a $ \alpha\in\C $ such that $ \sigma\cup\tau \subseteq \alpha. $  This implies $ \sigma\subset \alpha $ and $ \tau\subseteq \alpha. $ So, $ (\sigma,\alpha) $ and $ (\alpha,\tau) $ are edges in $ G_\C. $ Therefore $ \sigma \alpha\tau
	$ is a path in $ G_\C. $ Hence $ \sigma,\tau  $ is connected. And as $ \sigma,\tau $ is an arbitrary pair in $\C $ we get that $ G_\C $ is a connected graph.
\end{proof}
\begin{remark}\begin{enumerate}
		\item Note that if a codeword containment graph $G_\C$ satisfies the above condition then the distance between any two vertices $\sigma$ and $\tau$ is at most 2.  
		\item  Consider the code  $ \C=\{13,125,1235,1245\}. $ The CCG for this code can be seen in Figure \ref{ncccg}. The graph is connected, however it does not satisfy the above sufficient condition. Therefore, the converse is not true in general.   We leave it as a open question to check  when the converse is true.
	\end{enumerate}
\end{remark}

Next, we show that the surjective morphisms (Definition \ref{jeffmorphis}) preserve connectedness of the codeword containment graph.  Note that these morphisms are monotone (Proposition \ref{proptrunmomon}). 
\begin{proposition}
	Let $ f:\C\to\D $ be a morphism.  If $ G_\C$ is connected then $ G_{f(\C)} $ is connected. \label{thmorcon}
\end{proposition}
\begin{proof}
Suppose, $|f(\C)|=1$, then there is nothing to prove. So, we assume $|f(\C)|>1$.	Consider  $ f(\sigma) \not= f(\tau)\in f(\C).$ We show that there is a path between $ f(\sigma) $ and $ f(\tau). $ Note that $ \sigma,\tau\in\C. $ Also, as $ G_\C$ is connected there exists a path, say,  $ \gamma_0^{}\gamma_1^{}\cdots\gamma_m^{} $, between $ \gamma_0=\sigma $ and $ \gamma_m=\tau. $ Furthermore, either $ \gamma_i \subsetneq \gamma_{i+1} $ or $ \gamma_{i+1} \subsetneq \gamma_{i}. $ Also, as $ f $ is monotone, either $ f(\gamma_i) \subseteq f(\gamma_{i+1}) $ or $ f{(\gamma_{i+1})} \subsetneq f(\gamma_{i}). $ Therefore $ f(\gamma_0^{})f(\gamma_1^{})\cdots f(\gamma_m^{}) $ is a walk from $ f(\sigma) $ to $ f(\tau). $ Hence there exists a path between $ f(\sigma) $ and $ f(\tau). $ So, $ G_{f(\C)} $ is connected. 
\end{proof}
\begin{remark}
	Morphisms need not preserve disconnectedness of a codeword containment graph.
	For example, consider  $ \C =\{1,3,12\} $ and delete the 3rd neuron to get $ \C'= \{\emptyset,1, 12\}. $ Note that, deleting a neuron is a morphism. However $G_\C$ is disconnected and $G_{\C'}$ is connected. Look at Figure \ref{dcccg} for $G_\C$ and $G_{\C'}$.
	\end{remark}
\begin{figure}[h!]
	\centering
	\begin{subfigure}[b]{0.5\linewidth}
		\centering
		\begin{tikzpicture}  
			\vertex (1) at (0,0) [label=below:1] {};
			\vertex (12) at (1,0) [label=below:12]{};
			\vertex (3) at (0,1) [label=left:3]{};
			\path[-]
			(1) edge (12)
			;
		\end{tikzpicture}
		\caption{$G_\C$ when $\C=\{1,12,3\}$}
	\end{subfigure}
	\begin{subfigure}[b]{0.3\linewidth}
		\centering 
		\begin{tikzpicture}  
			\vertex (0) at (0,0) [label=below:$\emptyset$] {};
			\vertex (1) at (1,0) [label=below:1]{};
			\vertex (12) at (0,1) [label=left:12]{};
			\path[-]
			
			(0) edge (1)
			(0) edge (12)
			
			(1) edge (12)
			;
		\end{tikzpicture}
		\caption{$G_{\C'}$ when  $\C'=\{\emptyset,1,12\}$}
	\end{subfigure}
	\caption{}
	\label{dcccg}
\end{figure}
\begin{corollary}
		Let $ f:\C\to\D $ be an isomorphism.  Then $ G_\C$ is connected if and only if $ G_{f(\C)} $ is connected.
\end{corollary}
\begin{remark}
	As we know that all five elementary code maps described in Proposition \ref{thmnycc1} are also  morphisms we get that they all preserve connectedness via Theorem \ref{thmorcon}.
\end{remark}
\subsection{Complete CCG and morphisms}
In this subsection we will look at when the codeword containment graph is a complete graph. Recall that a graph $G$ is said to be complete if and only if every pair of distinct vertices of the graph are adjacent to each other. We will show that the surjective morphisms preserve the completeness of CCG. 

\begin{definition}[Complete code]
	A code $\C$ is said to be a \emph{complete code} if the codeword containment graph $G_\C$ of $\C$ is a complete graph. 
\end{definition}
\begin{eg}
	Some examples of complete codes are  $\{\emptyset,1,123,1234\}$, $\{\emptyset,1,12\}$,  $\{1,12,123\}$, etc.
	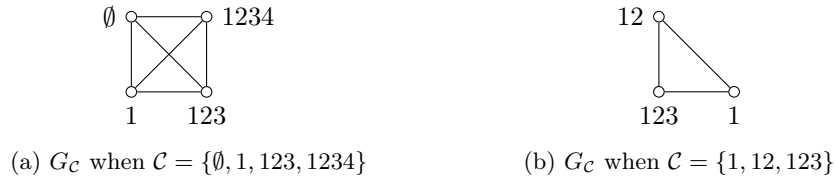
\begin{figure}[h!]
		\centering
		\begin{subfigure}[b]{0.5\linewidth}
			\centering
			\begin{tikzpicture}  
				\vertex (1) at (0,0) [label=below:1] {};
				\vertex (123) at (1,0) [label=below:123]{};
				\vertex (emp) at (0,1) [label=left:$\emptyset$]{};
				\vertex (1234) at (1,1) [label=right:1234]{};
				\path[-]
				(1) edge (emp)
				(1) edge (123)
				(1) edge (1234)
				(123) edge (1234)
				(123) edge (emp)
				(emp) edge (1234)
				;
			\end{tikzpicture}
			\caption{$G_\C$ when $\C=\{\emptyset,1,123,1234\}$}
		\end{subfigure}
		\begin{subfigure}[b]{0.3\linewidth}
			\centering 
			\begin{tikzpicture}  
				\vertex (0) at (0,0) [label=below:$123$] {};
				\vertex (1) at (1,0) [label=below:1]{};
				\vertex (12) at (0,1) [label=left:12]{};
				\path[-]
				
				(0) edge (1)
				(0) edge (12)
				
				(1) edge (12)
				;
			\end{tikzpicture}
			\caption{$G_{\C}$ when  $\C=\{1,12,123\}$} \label{comccg3}
		\end{subfigure}
		\caption{Complete codeword containment graph}
		\label{comccg}
	\end{figure}
	 
\end{eg}
Next, we have a remark on complete codes. 
\begin{remark}
		A code $\C$ is a complete code if and only if for all pairs $\sigma,\tau\in \C$ either $\sigma\subsetneq \tau $ or $\tau\subsetneq \sigma$. 
		Note that, the complete code along with inclusion is a total order. 
\end{remark}
We observe that the monotone property of a morphism (Refer to Proposition \ref{proptrunmomon}) gives us that it preserves complete codes.  We state the same in the next proposition. 
\begin{proposition}
	Let $ f:\C\to\D $ be a morphism. If $\C$ is a complete code then so is $f(\C)$. \label{thmorcom}
\end{proposition}
\begin{proof}
Let $f:\C \to \D$ be a morphism and $\C$ be a complete code. To show that $f(\C)$ is a complete code. If $\vert f(\C) \vert =1$ then we have nothing to prove. Let $f(\sigma)$ and $f(\tau)$ be any two distinct codes  $f(\C)$. We will show that either $f(\sigma)\subsetneq f(\tau)$ or $f(\sigma)\supsetneq f(\tau)$.  Note that we have $\sigma,\tau\in \C$ and since $\C$ is a complete code we get that either $\sigma\subsetneq \tau$ or $\sigma\supsetneq \tau$. Moreover, as $f$ is a morphism we use the monotone property (Refer \ref{proptrunmomon}) to get that  either $f(\sigma)\subsetneq f(\tau)$ or $f(\sigma)\supsetneq f(\tau)$. Therefore we get $f(\C)$ as a complete code. Hence the result. 
\end{proof}
\begin{remark}
	Note that, for given any $m\in \N$, there exist a complete code $\C$ with cardinality $m$. One such  code is $\{\emptyset,1,12,\dots,12\cdots m-1\}$, which we will denote as $\CC{m}$. Given a complete code with cardinality $m$ we have seen via examples that it is either $\CC{m}$ or it can be reduced to $\CC{m}$ via some composition of morphisms. For example, $\C=\{12,123,1234\}$ is a complete code and is isomorphic to $\CC{3}=\{\emptyset,1,12\}$. So, we claim that it is true in general, i.e.,  any complete code with cardinality $m$ is  isomorphic to $\CC{m}$. Next result will also justify the notation $\CC{m}$ where $\mathscr{C}$ is for complete code and $m$ stands for cardinality of $\C$.  
\end{remark} 
\begin{theorem} \label{concomccg}
	Let $\C$ be a complete code on $n$ neurons. If $|\C|=m$ then $\C$ is isomorphic to  $\CC{m}= \{\emptyset,1,12,\dots,12\cdots m-1\}.$ \label{thcomiso} 
\end{theorem}
\begin{proof}
 Let $\C=\{\sigma_1,\sigma_2,\dots\sigma_m\}$. Since $\C$ is a complete code, WLOG we can assume $\sigma_1\subseteq\sigma_2\subseteq \dots\subseteq\sigma_m.$ Define $f:\C \to \CC{m}$ such that $\sigma_i\mapsto 12\cdots i-1$. Note that $f(\sigma_1)=\emptyset$. Clearly, $f$ is a  well-defined bijective map.  We will show that both $f$ and $g=f^{-1}$ are morphisms. First we will note down simple trunks of both $\C$ and $\CC{m}$. The simple trunks of $\CC{m}$ are $ \Tk{\CC{m}}{i}=\{12\cdots i, 12\cdots ii+1,\dots, 12\cdots m-1\}$ for all $i\in[m-1]$. Before we write down the simple trunks of $\C $, note that given any $j\in[n]$ there exists a $k\in[m]$ such that $j\in\sigma_r$ for all $r\in[m]\backslash [k-1]$ and $j\notin \sigma_p$ for any $p\in[k-1]$. This comes from the fact that $\sigma_1\subseteq\sigma_2\subseteq \dots\subseteq\sigma_m$. So, simple trunks of $\C$ are $\Tk{\C}{j}=\{\sigma_{k},\sigma_{k+1},\dots,\sigma_m\}$ for all $j\in[n]$. 
 
 By Proposition \ref{propsimtrun} to show $f$ is a morphism, it is enough to show $f^{-1}(\Tk{\CC{m}}{i})$ is a trunk in $\C$ for all $i\in[m]$. Fix $i\in [m]$ and note that $f^{-1}(\Tk{\CC{m}}{i})=f^{-1}(\{12\dots i, 12\dots ii+1,\dots 12\dots m-1\}) =\{\sigma_{i+1},\sigma_{i+2},\dots\sigma_{m}\}=\Tk{\C}{l}$, where $l\in[n]$ such that $i+1$ is the least number such that $ l\in \sigma_{i+1}$. Therefore $f$ is a morphism.  Similarly one can show that $g=f^{-1}$ is also a morphism. Hence $f:\C\to \CC{m}$ is a isomorphism. 
 %Let $\vert \sigma_1\vert =k$, then $\vert\sigma_i \vert > k$ for all $i\in [m] \backslash \{1\}.  $ 
\end{proof}
\no  We will show that the complete codes are open convex in the next theorem. This gives another motivation for us to study complete codes.
\begin{theorem}
	Let $\C$ be a complete code then $\C$ is open convex. Moreover it has minimal open convex embedding dimension as 1. 
\end{theorem}
\begin{proof}
	%Next, when $m=2$ we have $\C_2=\{\emptyset,1\}$. So, consider $\U_2=\{U_1\}$, where $U_1 \subsetneq \R$ is any non-empty open convex set of $\R$. Then $\U_2$ is an open convex realization of $\C_2$ with $X=\R$. 
	For $m=1$, $\CC{1}=\{\emptyset\}.$ Trivially we consider $\U=\emptyset$ with $X=\R$ as the realization of $\CC{1}$. Next, for any $m\in\N\backslash\{1\}$, let  $\U_m=\{U_1,\dots U_{m-1}\}$ with $U_i$ as the open interval $(i,m)\subseteq \R$. Then we have $\U_m$ to be an open convex realization of $\CC{m}$ with $X=\R. $ Therefore $\CC{m}$ is an open convex code with minimal open convex embedding dimension as 1. 
	
	Jeffs in \cite[Theorem 1.3]{jeffs2020morphisms} showed that convexity and minimal embedding dimension
	are isomorphism invariants. So, for any complete code $\C$ with cardinality $m$,  Theorem \ref{thcomiso} tells us that $\C$ is isomorphic to $\CC{m}$.  Further as $\CC{m}$ is open convex, using Theorem 1.3 of \cite{jeffs2020morphisms} we get that $\C$ is an open convex code with minimal convex embedding dimension 1. 
\end{proof}
%\subsection{Codeword contained trees}
%explain few examples of trees. The directed ones. 
%Show there is an tree isomorphismssudo  but the codes dosent preserves NRH.  
\subsection{2-Regular CCG}
 Recall that a graph $G$ is said to be 2-regular if degree of every vertex is exactly 2. We will call a code $\C$ as \textit{2-regular code} if its CCG is a 2-regular graph. In this part of the section we will discuss which codes can never have a  2-regular codeword containment graph. Before that let us see few examples of 2-regular codes below. Note that the complete graph on 3 vertices descried above in Figure \ref{comccg3} is 2-regular. 
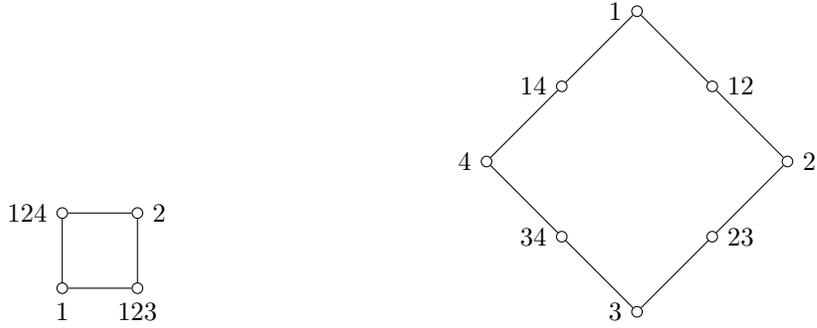
\begin{figure}[h!]
	\centering
	\begin{subfigure}[b]{0.5\linewidth}
		\centering
		\begin{tikzpicture}  
			\vertex (1) at (0,0) [label=below:1] {};
			\vertex (123) at (1,0) [label=below:123]{};
			\vertex (emp) at (0,1) [label=left:$124$]{};
			\vertex (1234) at (1,1) [label=right:2]{};
			\path[-]
			(1) edge (emp)
			(1) edge (123)
			
			(123) edge (1234)
			
			(emp) edge (1234)
			;
		\end{tikzpicture}
		\caption{$G_\C$ when $\C=\{1,2,123,1234\}$} \label{figreg4}
	\end{subfigure}
	\begin{subfigure}[b]{0.4\linewidth}
		\centering 
		\begin{tikzpicture}  
			\vertex (2) at (2,0) [label=right:$2$] {};
			\vertex (1) at (0,2) [label=left:1]{};
			\vertex (12) at (1,1) [label=right:12]{};
			\vertex (23) at (1,-1) [label=right:23]{};
	
			\vertex (34) at (-1,-1) [label=left:34]{};
			\vertex (3) at (0,-2) [label=left:3]{};
			\vertex (14) at (-1,1) [label=left:14]{};
			\vertex (4) at (-2,0) [label=left:4]{};
			\path[-]
			
			(2) edge (12)
			
			(1) edge (12)
			(1) edge (14)
			(4) edge (14)
			(3) edge (23)
			(3) edge (34)
			(4) edge (34)
			(2) edge (23)
			;
		\end{tikzpicture}
		\caption{$G_{\C}$ when  $\C=\{1,2,3,4,12,23,34,14\}$}
	\end{subfigure}
	\caption{2-regular codeword containment graph}
	\label{conccg}
\end{figure}
	\begin{proposition}
	Let $\C$ be a code on $n$ neurons with $\vert \C\vert = m>3$. If $G_\C$ is connected 2-regular then $m$ is even. 
\end{proposition}
\begin{proof}
	Let $\C$ be a code on $n$ neurons with $\vert\C\vert=m>3$. Let $\C=\{\sigma_1,\sigma_2,\sigma_3,\dots,\sigma_m\}$. Assume $m$ to be odd. Let $G_\C$ be a connected 2-regular graph.  So, for any $i,j\in[n]$ we have that $\sigma_i$ and $\sigma_j$ have a path between them. Further for any $i\in[n]$, $\sigma_i$ has exactly degree 2. Therefore, without loss of generality we can relabel the codewords of $\C$ to get $\sigma_1\subsetneq \sigma_2\supsetneq \sigma_3\subsetneq \sigma_4\cdots \subsetneq\sigma_{m-1}\supsetneq\sigma_m$. Further since the graph is 2-regular we either have $\sigma_m\supsetneq\sigma_1$ or $\sigma_m\subsetneq\sigma_1$. However, this implies   $\sigma_1\subsetneq \sigma_{m-1}$ or $\sigma_m\subsetneq \sigma_2$, respectively. This leads to a contradiction to the fact that the  graph $G_\C$ is 2-regular. Hence $m$ must be even.  \end{proof}
	
\begin{remark}
    For $k\in\N \setminus \{1\}$ we note that there always exists at least one code $\C$ with $\vert \C \vert = 2k $ such that its CCG is a 2-regular graph. We give these code for all such $k$'s. When $k=2$, the code $\{1,2,123,1234\}$ is a 2 regular code with cardinality $2k=4$ (Refer Figure \ref{figreg4}). Next, for all $k\in \N \setminus \{1,2\}$ we can check that the codes $ \CR{k}:=\{1,2,\dots, k, 12,23,\dots, k-1k,k1\}$ has its CCG as 2-regular. Furthermore, there exists even cardinality codes with its CCG as 2-regular and not isomorphic to $\CR{k}$ for some $k$. We will discuss this further in Remark \ref{remcr6}, for which we will use the fact that  $\CR{k}$ is both open and closed convex. We will prove this in the following proposition.

%Furthermore, we give an example that there are 2-regular codes that  that given any 2-regular code with even cardinality greater than four is isomorphic to such code.  We end this section with the following the example. 
	\end{remark}
	\begin{proposition}
	Let $k\in\N\backslash\{1,2\}$ then $\CR{k}$ is both open and closed convex. Moreover, we show that the minimal open and closed convex embedding dimension is 2.  \label{propclosedcr}
	\end{proposition}
	\begin{proof}
	Fix $k\in\N\backslash\{1,2\}$. We will first construct a closed convex realization of $\CR{k}$. Consider $k$ points in $\R^2$  and label them as $\{1,\dots,k\}.$ Further, we make sure that the distance between $i$ and $i+1\mod k$ is equal for all $i$. Next, draw a $k$-gon using these points as vertices. Now, consider $U_i$ to be the edge from $i$ to $i+1\mod k$. We claim that $\U=\{U_1,\dots,U_k\}$ is a closed convex realization of $\CR{k}$ with $X=\bigcup_{i\in[k]} U_i$. Clearly, $U_i$'s are non empty closed convex sets in $\R^2$. Observe that $U_i\cap U_j\not =\emptyset$ iff $i-j \cong 1 \mod
	 k$. Therefore we have that $\U$ is a closed convex realization of $\CR{k}$ in $\R^2$. Further, using Proposition 2.2 of  \cite{suhith2021few} we see that $\CR{k}$ cannot have a convex realization in $\R$. Hence the minimal closed convex embedding dimension is 2. 
	 
	 The open convex realization of $\CR{k}$ is given in Figure \ref{figrealcrk}. Note that this is a realization in $\R^2$ with $X=\bigcup_{i\in[k]} U_i$. Once again using Proposition 2.2 of  \cite{suhith2021few}, we get that the  minimal open convex embedding dimension is 2. 
	 \begin{figure}
	 	\centering
	 	\begin{tikzpicture}[scale=0.5]
	 		\draw[dashed] (0,0) rectangle  (3,1);
	 		\draw[dashed,rotate=35] (2,-1.5) rectangle  (5,-0.5);
	 		\draw[dashed,rotate=120] (0,-1) rectangle  (3,0);
	 			\draw[dashed] (3.8,1.5) rectangle  (4.8,4.5);
	 			\draw(4,5.3) node {$\ddots$};
	 			\draw[dashed] (3.8,1.5) rectangle  (4.8,4.5);
	 			\draw(3.3,6.5) node[left] {$\ddots$};
	 			\draw(0.5,6.5) node[right] {$\rddots$};
	 			\draw(-0.8,5.5) node[right] {$\rddots$};
	 			\draw(-1.5,4) node[right] {$\vdots$};
	 			\draw(0.7,0.5) node[right] {$\small U_1$};
	 			\draw(3,1.3) node[right] {$\small U_2$};
	 			\draw(3.7,3.3) node[right] {$\small U_3$};
	 			\draw(-1.1,2) node[right] {$\small U_k$};
	 	\end{tikzpicture}
	 	\caption{Open convex realization of $\CR{k}$} \label{figrealcrk}
	 \end{figure}
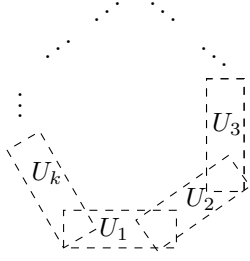
	\end{proof}
%\begin{figure}
%		\begin{subfigure}[b]{0.5\linewidth}
%			\begin{tikzpicture}[scale=0.5]
%				\draw (1,0) ellipse (3cm and 0.7cm);
%				\draw[dashed] (2,1) ellipse (0.7cm and 3cm);
%				\draw (0,1) ellipse (0.7cm and 3cm);
%				\draw[rotate=45] (3,2) ellipse (2cm and 0.7cm);
%			\end{tikzpicture}
%		\end{subfigure}
%\end{figure}
\begin{remark} \label{remcr6}
	Note that there are $2$ regular codes which may not be isomorphic to $\CR{k}$. We discuss one such example. Let $k=6$, consider $\CR{6}=\{1,2,3,4,5,6,12,23,34,45,56,16\}$. We know that $ \CR{6}$ is closed convex with a realization that can be obtained as discussed in the proof of Proposition \ref{propclosedcr}. 
Further, consider the code $\C=\{ 12, 16, 56, 45, 34, 23,123, 126, 156, 456, 345, 234\}$. One can check that $\C$ is 2-regular code. However, $\C$ is \textit{not} a closed convex code (Refer \cite[Lemma 2.9]{cruz2019open}). So, if there is an isomorphism from $\C$ to $\CR{6}$ then it violates Theorem 1.3 of \cite{jeffs2020morphisms} (isomorphisms preserve closed convexity). Hence, $\C$ cannot be isomorphic to $\CR{6}$. 
\end{remark}
 
Next we will work with the general relationship graph $G(\C)$ corresponding to a neural code $\C$. This graph is defined using the canonical form of a neural ideal of $\C$. We will first understand the definitions of neural ideal and its canonical form in the next section. Then in section \ref{sec: grg} we finally define general relationship graphs for any code $\C$ and look into some of its properties. 
\section{Canonical forms and their behavior under elementary code maps } \label{sec:cf}

In this section we define pseudo-monomial ideals and their canonical forms. Then for any given neural code $\C$, we define a neural ideal, denoted by $\J_\C$. We will also give a detailed explanation of algorithm to generate canonical form of neural ideals which was described by Nora Youngs \cite{youngs2014neural}.  Lastly, we will discuss the behaviour of canonical forms under elementary code maps. 
\subsection{Neural ideal and its canonical form}
\begin{definition}\cite{curto2013neural} We discuss the definitions pseudo-monomials, pseudo-monomial ideal and minimal pseudo-monomial.	\begin{enumerate}
		\item If $ f\in \mathbb{F}_2[x_1,\dots,x_n] $ has the form $ f=\prod_{i\in\sigma}x_i\prod_{j\in\tau} x(1-x_j) $ for some $ \sigma,\tau\subseteq [n] $ with $ \sigma \cap \tau=\emptyset, $ then we say that $ f $ is a \textit{pseudo-monomial}.
		\item An ideal $ J\subset \mathbb{F}_2[x_1,\dots,x_n]  $ is a pseudo-monomial ideal if $ J $ can be generated by a finite set of pseudo-monomials.
		\item Let $ J\subset \mathbb{F}_2[x_1,\dots,x_n]$ be an ideal, and  $ f\in J $ a pseudo-monomial. We say that $ f $ is a minimal pseudo-monomial of $ J $ if there does not exist another pseudo-monomial $ g\in J $ with $ deg(g)< deg(f) $ such that $ f=hg $ for some $ h\in  \mathbb{F}_2[x_1,\dots,x_n]. $
\end{enumerate}\end{definition}
\begin{definition}[Canonical form] \cite{curto2013neural} A pseudo-monomial ideal, $ J $ is said to be in canonical form if $ J =\langle f_1,\dots,f_l\rangle,$ where the set $ \operatorname{CF}(J)=\{f_1,\dots,f_l\} $ is the set of all minimal pseudo-monomials of $ J $. We refer to $ \operatorname{CF}(J) $ as the canonical form of $ J. $
\end{definition}

\no Let $ v\subseteq [n] $ then define $ \rho_v $ a pseudo-monomial in $  \mathbb{F}_2[x_1,\dots,x_n] $ as $ \rho_v= \prod_{i\in v}x_i\prod_{j\notin v} x(1-x_j)$. Now we define neural ideal for a given code $\C$. 

\begin{definition}[Neural ideal]
	Let $ \C $ be a neural code on $ n $ neurons then the ideal $ \mathcal{J}_\C\subseteq \mathbb{F}_2[x_1,\dots,x_n]$ generated by all the pseudo-monomials $ \rho_v  $, for $ v\notin \C $ is called the neural ideal, i.e. $ \mathcal{J}_\C=\ideal{\{\rho_v \mid v\notin \C\}} $.
\end{definition}
\begin{eg}
	Let $ \C=\{12,23,\emptyset\} $ then $ \J_\C=\langle\rho_{1},\rho_{13},\rho_{123},\rho_{2},\rho_{3}\rangle  $. We compute $ \CF{\C} $ using sage math by using the algorithm given by Petersen et al. \cite{petersen2018neural} and we obtain $ \CF{\C}= \{x_1(1-x_2), x_1x_3,(1-x_2)x_3,x_2(1-x_3)(1-x_1)\}. $
\end{eg}
\no Next, we discuss an algorithm to generate canonical form of neural ideals given by Youngs \cite{youngs2014neural}.
\subsection{Algorithm to generate canonical form of neural ideals}\label{youngsalgo}
This algorithm can be found in Nora Youngs thesis \cite{youngs2014neural}. Also, S Magallanes has given a detailed explanation of the same in her academic report    \cite{magallanes2019neural}.  Let $\C$ be a code on $n $ neurons and  $ \C=\{c_1,c_2\dots, c_m\} $. Define the \textit{binary form} for $ c_i $ as the vector $ c_{i1}^{}c_{i2}^{}\dots c_{in}^{}\in\{0,1\}^n $ such that $ c_{ij}=1 $ if and only if $ j\in c_i.$ The following are the detailed steps to obtain $\CF{\C}.$
 \begin{enumerate}[Step 1:]
	\item For all $ i\in[m] $ define $ P_{c_i} $ to be the ideal of $ \mathbb{F}_2[x_1,x_2\dots,x_n] $ generated by $ \{x_1-c_{i1}^{},x_2-c_{i2}^{},\dots, x_n-c_{in}^{}\} $, i.e., $ P_{c_i}= \ideal{x_1-c_{i1}^{},x_2-c_{i2}^{},\dots, x_n-c_{in}^{}}. $ Note that in $ \mathbb{F}_2[x_1,x_2\dots,x_n] $ we have $ x_j-1=1-x_j $ and we will replace the same everywhere. 
	%	\item Notice that there are $ n $ generators for each ideal and in total there are $ m $ ideals. We now denote $ GM(\J_\C) $ to be a $ m\times n $ matrix with $ ij^{\text{th}} $ entry of this matrix being the $ j^{\text{th}} $ generator of the $ i{\text{th}} $ ideal $ P_{c_i} $ i.e., $ (GM((\J_\C)))_{ij}=x_j-c_{ij}. $
	\item We form the generating set for the product of these ideals ($ P_{c_i} $). Define $$ M(\nei{\C}) =\left\{\prod_{i=1}^{m}g_i\mid g_i \text{ is  a generator of the ideal } P_{c_i}\right\}. $$ We will write this definition more technically. First denote, $ F(X,Y) $ to be the set of all functions from $ X $ to $ Y. $ We get 
	$ M(\nei{\C})=\left\{\prod_{i=1}^{m}(x_{\tau(i)}^{}-c_{i\tau(i)}^{})\mid \tau \in F([m],[n])\right\}. $
	\item Impose the relation $ x_i(1-x_i)=0 $. So, $ x_i^2=x_i $ and $ (1-x_i)^2=(1-x_i). $ Further remove all the redundant elements. As 0 doesn't encode any useful information, we remove it too. Call the new set obtained as $ \widehat{M}(\nei{\C}). $
	\item Remove each pseudo-monomial in $ \widehat{M}(\nei{\C}) $ that is a multiple of another pseudo-monomial in the set. We denote this set to be $ \widetilde{M}(\nei{\C}). $

\end{enumerate}
\no Then $ \widetilde{M}(\nei{\C}) =\CF{\C}.$ For further details please refer to  \cite[Proposition 1]{youngs2014neural}.

\no In the next section we will consider $q:\C\to\D$ to be a elementary code map as mentioned in Theorem \ref{thmnycc1} and provide the relationship between $\CF{\C}$ and $\CF{\D}$. 

\subsection{Behaviour of canonical forms under elementary code maps}
Given two codes $\C$ and $\D$ on $n$ and $m$ neurons respectively, it now becomes important to understand the relation between $\J_\C$ and $\J_\D$. Jeffs, Omar and Youngs \cite{jeffs2018homomorphisms} look into polynomial ring homomorphisms between $\mathbb{F}_2[x_1,\dots,x_n] \to \mathbb{F}_2[y_1,\dots,y_m] $ that preserve neural ideals. They demonstrate how all homomorphisms of this kind can be divided into three basic types. These morphisms were permutation, bit flip and a restriction map. In our paper we work with neural ring homomorphisms between $\ring{\D}\to \ring{\C}$ and give the relation between $\CF{\D}$ and $\CF{\C}.$  We will use this algorithm discussed in section \ref{youngsalgo} to prove our results.
\begin{theorem}\label{thrperm}
	Let $\C,\D$ be two codes on $n$ neurons with $q:\C\to \D$ being a permutation map, via the permutation $\gamma\in S_n$. Then $\CF{\D}=\gamma(\CF{\C}),$ i.e., $\CF{\D}=\{x_{\gamma(\sigma) }\mid x_\sigma\in \CF{\C}\}.$
\end{theorem}
\no The proof of this above theorem comes from the proof of Theorem 2.18 of \cite{jeffs2018homomorphisms}. 
\begin{theorem} \label{thraddon}
	Let $ \C, \D $ be two codes on $ n, n+1 $ neurons respectively with $ q:\C\to \D $ being a adding trivial on neuron. Then $ \CF{\D}=\CF{\C}\cup\{1-x_{n+1}\}. $
\end{theorem}
\begin{proof}
	Observe that  $ M(\mathcal{J}_\C) \subseteq   M(\mathcal{J}_\D)  $ as $ F([m],[n])\subseteq F([m],[n+1]). $ Therefore $ \CF{\C}\subseteq \CF{\D}. $  Moreover, as the last neuron is always on, so $ c_{in+1}=1 $ for all $ i\in[m]. $  Let   $ \gamma\in F([m],[n+1]) $ be the function such that $ \gamma(i)=n+1 $ for all $ i $. As  $ \prod_{i=1}^{m} (x_{\tau(i)}^{}-c_{i\tau(i)}^{})\in M(\nei{\D}) $ we get  $ \prod_{i=1}^{m} (x_{n+1}^{}-c_{in+1}^{})=\prod_{i=1}^{m}(x_{n+1}^{}-1)=(1-x_{n+1}^{})^m\in M(\nei{\D}) $ when $ \tau=\gamma. $ This guarantees that the pseudo-monomial $ (1-x_{n+1}) $ appears in $ \CF{\D}. $  Also, note that any other pseudo-monomials involving $ x_{n+1} $ will be redundant. Hence the proof. 
\end{proof}
\begin{theorem}\label{thraddoff}
	Let $ \C, \D $ be two codes on $ n, n+1 $ neurons respectively with $ q:\C\to \D $ being a adding trivial off neuron. Then $ \CF{\D}=\CF{\C}\cup\{x_{n+1}\} $.
\end{theorem}
\no Proof of this theorem is similar to the previous one.  
\begin{theorem}\label{thrdup}
	Let $ \C, \D $ be two codes on $ n, n+1 $ neurons respectively with $ q:\C\to \D $ be the map that duplicates the $ i^{th} $ neuron. Then $ \CF{\D}=\CF{\C}\cup\{\alpha(x_{n+1}/x_i)\mid \alpha\in\CF{\C} \text{ and } x_i\vert \alpha\} \cup \{\alpha(1-x_{n+1})/(1-x_i)\mid \alpha\in\CF{\C} \text{ and } (1-x_i)\vert \alpha\} \cup \{x_i(1-x_{n+1}),(1-x_i)x_{n+1}\} $.
\end{theorem}
\begin{proof}
	Observe that here too $ M(\mathcal{J}_\C) \subseteq   M(\mathcal{J}_\D)  $ as $ F([m],[n])\subseteq F([m],[n+1]). $ Therefore $ \CF{\C}\subseteq \CF{\D}. $  Moreover, as the last neuron is the $ i^{\text{th}} $ neuron, i.e., $ c_{kn+1}=c_{ki} $ for all $ k\in[m]. $  Choose  $ \gamma\in F([m],[n]) $ such that there exists $ k\in[m] $ such that $ \gamma(k)=i $. Then define  $ \gamma'\in F([m],[n+1]) $  such that $$ \gamma'(k)=\begin{cases}
		\gamma(k) \quad \text{ if } \gamma(k)\not=i \\
		n+1 \quad \text{ if } \gamma(k) =i.
	\end{cases} $$ Notice that the pseudo-monomials that were involved in $ M(\nei{\C}) $ with $ x_i $ or $ 1-x_i $ term present in it (say $ \alpha $) comes from the function $ \gamma. $ Also, now $ \gamma' $ replaces $ x_i $ with $ x_{n+1} $ in $ \alpha $ and we get that this new pseudo-monomial is present in $ M(\nei{\D}). $ Rest of the result follows from the algorithm. 
\end{proof}
\begin{theorem}\label{thrproj}
	Let $ \C, \D $ be two codes on $ n, n-1 $ neurons respectively with $ q:\C\to \D $ being a projection map (delete the $ n^{\text{th}} $ neuron). Then $ \CF{\D}=\{\alpha\in\CF{\C}\mid x_{n}\not\vert\alpha \text{ and } (1-x_n)\not\vert\alpha\}. $
\end{theorem}
\begin{proof}
	Notice that $ F([m],[n-1])\subseteq F([m],[n]) $. Hence we get $ \CF{\D}\subseteq \CF{\C}. $ Let  $E= \{\alpha\in F([m],[n]) \mid \text{  there exists } k\in[m] \text{ with } \alpha(k)=n \}$ . Clearly $ F([m],[n-1])= F([m],[n])\backslash E. $ So, the result follows. 
\end{proof}
\section{General relationship graph (GRG)} \label{sec: grg}
In this section, we first see the definition of general relationship graph for degree two codes as given by Gross et al. \cite{gross2018neural}. Later, we will define a simplicial complex called as general relationship complex for a given code. Using this complex we will extend the definition of general relationship graph to every code. Later, we will study the behavior of general relationship graphs under the elementary code maps of Theorem \ref{thmnycc1}.
	\begin{definition}[Degree two codes]
	A code $ \C $ is said to be of degree two  if every pseudo-monomial in $ \CF{\C} $ has degree exactly two. 
\end{definition}
\begin{definition}[General relationship graph for degree two codes]\cite{gross2018neural}
	Let $ \C $ be a degree two code on $ n $ neurons then define $ G(\C) $ to be the general relationship graph with vertex set $ V=[n] $ and an edge $ {i,j} $ appears if and only if $ \CF{\C} $ doesn't contain a two variable pseudo-monomial whose two variable are $ x_i $ and $ x_j $ i.e.,  none of  $ x_ix_j, x_i(1-x_j), x_j(1-x_i) $ and $ (1-x_i)(1-x_j) $ belongs to $ \CF{\C}. $  
\end{definition}

\begin{eg}
	Let $ \C=\{\emptyset,1,2,3,4,12,14,23,24\} $. Then computing $\CF{\C}$ using sage math algorithm given by Peteresen et al.\cite{petersen2018neural} $ \CF{\C}=\{x_1x_3,x_2x_4\}. $ Figure \ref{figgcs} represents $ G(\C). $
\end{eg}
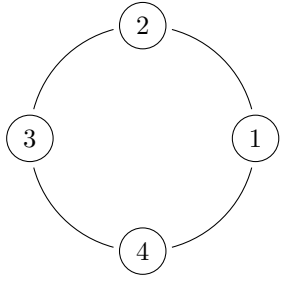
\begin{figure}
	\centering
	\begin{tikzpicture}
		
		\def \n {4}
		\def \radius {1.5cm}
		\def \margin {15} % margin in angles, depends on the radius
		
		\foreach \s in {1,...,\n}
		{
			\node[draw, circle] at ({360/\n * (\s - 1)}:\radius) {$\s$};
			\draw[-, >=latex] ({360/\n * (\s - 1)+\margin}:\radius) 
			arc ({360/\n * (\s - 1)+\margin}:{360/\n * (\s)-\margin}:\radius);
		}
	\end{tikzpicture}
	\caption{$ G(\C) $ when $\C=\{\emptyset,1,2,3,4,12,14,23,24\}$.}
	\label{figgcs}
\end{figure}
\no Now, we will define general relationship complex. Before that denote, $E_\sigma=\{x_i,1-x_i\mid i\in\sigma\},$ for every $\sigma\subseteq [n]$. 
\begin{definition}[General relationship complex]
	Let $\C$ be a code on $n$ neurons. Then g\textit{General relationship complex} is denoted by $GR(\C)$ and is a sub-complex of $\mathcal{P}([n])$\footnote{$\mathcal{P}(A)$ denotes the power set of $A$. Moreover it is clearly a simplicial complex. } satisfying,
	$$GR(\C)=\left\{\sigma \subseteq [n] \ \Big\vert \prod_{\gamma\in\Gamma} \gamma \notin \CF{\C} \text{ for any } \Gamma\subseteq E_\sigma   \right\}.$$ 
\end{definition}	

\begin{eg}
	In this example we will consider various $\J_\C$ and write down their respective $GR(\C)$. \label{exgrc}
	\begin{enumerate} 
		\item $\CF{\C}=\{x_1x_2x_3,x_1x_2\}$ then  $ GR(\C)=\Delta(\{13,23\}).$ Refer to Figure \ref{figgrca}.   
		\item $\CF{\C}=\{x_1x_2,x_1x_3\}$ then $GR(\C)= \Delta{(\{1,23\})}$. Refer to Figure \ref{figgrcb}.
		\item  $\CF{\C}=\{x_1x_2,x_2x_4\}$ then $GR(\C)=\Delta(\{134,23\})$. Refer to \ref{figgrcc}. \label{exgrg=1sgrc}
	\end{enumerate}
	
\end{eg}

\begin{figure}[h!]
	\centering
	\begin{subfigure}[b]{0.2\linewidth}
		\centering
		\begin{tikzpicture}  
			\vertex (1) at (0,0) [label=below:1] {};
			\vertex (2) at (1,0) [label=below:3]{};
			\vertex (3) at (2,0) [label=below:2]{};
			
			\path[-]
			(1) edge (2)
			(2) edge (3)
			;
		\end{tikzpicture} \caption{}
		\label{figgrca}
	\end{subfigure}
	\begin{subfigure}[b]{0.2\linewidth}
		\centering
		\begin{tikzpicture}  
			\vertex (1) at (0,0) [label=below:1] {};
			\vertex (2) at (1,0) [label=below:2]{};
			\vertex (3) at (2,0) [label=below:3]{};
			
			\path[-]
			
			(2) edge (3)
			;
		\end{tikzpicture}\caption{}\label{figgrcb}\end{subfigure}
	\begin{subfigure}[b]{0.2\linewidth}
		\centering
		\begin{tikzpicture}[scale=0.5]
			\tikzstyle{point}=[circle,thick,draw=black,fill=black,inner sep=0pt,minimum width=4pt,minimum height=4pt]
			\node (a)[point, label={[label distance=-0.7cm]:$4$}] at (0,0) {};
			\node (b)[point,label={[label distance=0cm]5:$1$}] at (3,0) {};
			\node (c)[point,label={[label distance=-0.6cm]4:$3$}] at (2,2) {};

			\node (d)[point,label={[label distance=-0.7cm]2:$2$}] at (0,2) {};
			%			\node (f)[point,label={[label distance=0cm]2:$6$}] at (4,2) {};

			%\node (p)[point,label={[label distance=0cm]5:$P$}] at (1.5,0.5) {};
			
			%\draw[pattern=north east lines] (a.center) -- (p.center) -- (b.center) -- cycle;
			\draw[pattern=north west lines] (a.center) -- (b.center) -- (c.center) -- cycle;
			%\draw[pattern=vertical lines]   (b.center) -- (p.center) -- (c.center) -- cycle;
			%		\draw[pattern=dots] (d.center) -- (e.center) -- (f.center) -- cycle;
			%	\draw (a.center) -- (d.center);
			%	\draw (e.center)-- (d.center);
			\draw (d.center)--(c.center);
		\end{tikzpicture}
		\caption{}
		\label{figgrcc}
	\end{subfigure}
	\caption{}
\end{figure}
	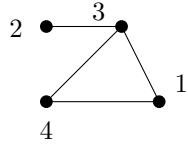
\begin{figure}[h]
	\centering
	\begin{tikzpicture}[scale=0.5]
		\tikzstyle{point}=[circle,thick,draw=black,fill=black,inner sep=0pt,minimum width=4pt,minimum height=4pt]
		\node (a)[point, label={[label distance=-0.7cm]:$4$}] at (0,0) {};
		\node (b)[point,label={[label distance=0cm]5:$1$}] at (3,0) {};
		\node (c)[point,label={[label distance=-0.6cm]4:$3$}] at (2,2) {};

		\node (d)[point,label={[label distance=-0.7cm]2:$2$}] at (0,2) {};
		%			\node (f)[point,label={[label distance=0cm]2:$6$}] at (4,2) {};

		%\node (p)[point,label={[label distance=0cm]5:$P$}] at (1.5,0.5) {};
		
		%\draw[pattern=north east lines] (a.center) -- (p.center) -- (b.center) -- cycle;
		%	\draw[pattern=north west lines] (a.center) -- (b.center) -- (c.center) -- cycle;
		%\draw[pattern=vertical lines]   (b.center) -- (p.center) -- (c.center) -- cycle;
		%		\draw[pattern=dots] (d.center) -- (e.center) -- (f.center) -- cycle;
		\draw (a.center) -- (c.center);
		\draw (a.center) -- (b.center);
		\draw (b.center)-- (c.center);
		\draw (d.center)--(c.center);
	\end{tikzpicture}
	\caption{$G(\C)$ when $\CF{\C}=\{x_1x_2,x_2x_4\}$. }
	\label{figgrgd}
\end{figure}
\begin{remark}
	Further in Example \ref{exgrc}(\ref{exgrg=1sgrc}), note that $\C$ is a degree two code and $G(\C)$ can be seen in Figure \ref{figgrgd}. So we get that the 1-skeleton of $GR(\C)$ is $G(\C).$ This example motivates us to define general relationship graph for any code as follows: 
	
\end{remark}

\begin{definition}[General relationship graph for any degree code]
	Given any code $\C$ we define general relationship graph, $G(\C)$ to be the 1-skeleton of $GR(\C).$
\end{definition}
\no Note that, this definition is a clear extension of general relationship graphs of degree two codes.  
\begin{remark}
	Note that when $\CF{J_\C}$ does not contain any two degree codeword then $G(\C)$ is complete graph. 
\end{remark}
\subsection{Behaviour of general relationship graph under elementary code maps}
Let $\C$ and $\D$ be two codes with $q:\C\to \D$ as an elementary code map.  Let $V(G),E(G)$ be the vertex and edge set of a graph $G$, respectively. We will describe $V(G(\D)) $ and $E(G(\D))$, in terms with  $V(G(\C)) $ and $E(G(\C)),$ respectively, for various elementary code maps. 
\begin{enumerate}
	\item \textbf{Permutation:} In this case, let $\C$ be a code on $n$ neurons and $\D$ be a code obtained after permuting every codeword of $\C$ by a permutation $ \gamma\in S_n$. From Theorem \ref{thrperm},  we get the following data about the general relationship graph of $\D$. \begin{enumerate}
		\item $V(G(\D))=\gamma(V(G(\C)))$, i.e., $V(G(\D))=\{x_{\gamma(i)}\mid x_i\in V(G(\C))\}$,
		\item  $E(G(\D))=\gamma(E(G(\C))),  $ i.e., $  E(G(\D))= \{x_{\gamma(i)}x_{\gamma(j)} \mid x_ix_j\in E(G(\C)) \}.$ 
	\end{enumerate}
	Clearly, $G(\C)$ and $G(\D)$ are isomorphic as graphs, since the canonical forms of $\J_\C$ and $\J_\D$ are bijective via the permutation $\gamma$, i.e., $\CF{\C}=\gamma(\CF{\D}).$ 
	\item \textbf{Adding a trivial neuron:} In this case, let $\C$ be a code on $n$ neurons and $\D$ be a code obtained after adding a trivial neuron as the  $n+1^{\text{th}}$ neuron for every codeword in $\C$. From Theorem \ref{thraddon} and Theorem \ref{thraddoff},  we get the following data about the general relationship graph of $\D$. \begin{enumerate}
		\item $V(G(\D))=V(G(\C))$,
		\item  $E(G(\D))=E(G(\C)). $
	\end{enumerate}
	Note that the vertex set does not have the vertex ${x_{n+1}}$ as $x_{n+1}$ or $1-x_{n+1}$ appear in $\CF{\D}$. SO, by definition of general relationship complex we do not have $x_{n+1}$ in it. Hence the $G(\D)$ is exactly same as $G(\C)$. 
		\item \textbf{Adding a duplicate  neuron:} In this case, let $\C$ be a code on $n$ neurons.  Let $ i^{\text{th}}$ neuron be duplicated and added as the $n+1^{\text{th}}$ neuron to obtain $\D$. From Theorem \ref{thrdup}, we get the following data about the general relationship graph of $\D$. \begin{enumerate}
		\item $V(G(\D))=V(G(\C))\cup \{x_{n+1}\}$,
		\item  $E(G(\D))=E(G(\C))\cup \{x_jx_{n+1}\mid x_jx_i\in E(G(\C))\}. $
	\end{enumerate}
		\item \textbf{Projection:} Let $\C$ be a code on $n$ neurons and delete the $n^{\text{th}}$ neuron from every codeword to obtain $\D$. So $\D$ is a code on $n-1$ neurons. From Theorem \ref{thrproj} we get the following data about the general relationship graphs of $\D$. \begin{enumerate}
		\item $V(G(\D))=V(G(\C))\setminus \{x_n\}$,
		\item  $E(G(\D))=E(G(\C))\setminus \{x_ix_{n} \in E(G(\C))\mid i\in [n]\}. $
	\end{enumerate}
\end{enumerate} 
\no We conclude this section with an open question to find the behaviour of canonical forms and general relationship graphs under  morphism. In the last section, we compare GRG and CCG for certain codes. 
\section{Comparing GRG and CCG for certain codes}
In this section we will look at GRG for complete and 2-regular codes. We will then compare these graphs to CCGs and study their behaviour.  
\subsection{GRG for complete codes}
Recall for any $m\in \N$, $\C_m=\{\emptyset,1,12,\dots,12\cdots m-1\}$ is a complete code with cardinality $m$. We will first study $G(\C_m)$ in this section. Using Sage math algorithm given by Petersen et al. \cite{petersen2018neural} we obtain $\CF{\C_m}$ for some $m$. They are as follows: 
\begin{enumerate}
	\item $\CF{\C_3}=\{x_2\cdot(1-x_1)\}$
	\item $\CF{\C_4}=\left\{x_{2} \cdot (1-x_{1}), x_{3} \cdot (1-x_{1}), x_{3} \cdot (1-x_{2})\right\}$
	\item $\CF{\C_5}= \left\{x_{2} \cdot (1-x_{1}), x_{3} \cdot (1-x_{1}), x_{3} \cdot (1-x_{2}), x_{4} \cdot (1-x_{1}), x_{4} \cdot (1-x_{2}), x_{4} \cdot (1-x_{3})\right\}$
	\item $\CF{\C_6}= \{x_{2} \cdot (1-x_{1}), x_{3} \cdot (1-x_{1}), x_{3} \cdot (1-x_{2}), x_{4} \cdot (1-x_{1}), x_{4} \cdot (1-x_{2}), x_{4} \cdot (1-x_{3})  x_{5} \cdot (1-x_{1}), x_{5} \cdot (1-x_{2}), x_{5} \cdot (1-x_{3}), x_{5} \cdot (1-x_{4})\}$
\end{enumerate} Observe that for $ m\in \{3,4,5,6\}$ we have $\C_m$ is a degree two code and clearly $G(\C_m)$ is a totally disconnected graph\footnote{A graph with empty edge set is called totally disconnected graph.}. In other-words the complement of $G(\C_m)$ is a complete graph for $m\in\{3,4,5,6\}$. We further claim this is true for all $m\in \N\backslash\{1,2\}.$

\begin{lemma} \label{lemmacfcm}
	Let $\C_m=\{\emptyset,1,12,\dots,12\cdots m-1\}$ for any $ m\in \N \backslash \{1,2\}$. Then $\CF{\C_m}=\{x_i(1-x_j) \mid \text{ for all } i,j \in [m] \text{ with } i> j \}.$
\end{lemma} 
Before we prove the above lemma, we will note down a theorem given by Curto et al. \cite{curto2013neural} to obtain the canonical form of a code given one of its open realization.  
\begin{theorem}\cite[Theorem 4.3]{curto2013neural}
	Let $\C$ be a code on $n$ neurons, and let $\U=\{U_1,\dots,U_n\}$ be any collection of open sets (not necessarily convex) in a nonempty stimulus space $X$ such that $\C=\C(\U)$. The canonical form of $\J_\C$ is:  \label{thmcurtcan}
	\begin{align*}
		\J_\C = & \langle \{x_\sigma \mid \sigma \text{ is minimal w.r.t } U_\sigma=\emptyset\}, \\ & \{ x_\sigma\prod_{i\in\tau}^{}(1-x_i) \mid \sigma,\tau \not = \emptyset, \sigma\cap\tau=\emptyset, U_\sigma \not= \emptyset,\bigcup_{i\in\tau} U_i \not = X, \text{ and } \sigma,\tau \text{ are each minimal w.r.t. } \\ & U_\sigma \subseteq \bigcup_{i\in\tau} U_i \},  \{ \prod_{i\in\tau}(1-x_i) \mid \tau \text{ is minimal w.r.t } X \subseteq \bigcup_{i\in\tau} U_i \} \rangle. 
	\end{align*}
\end{theorem}
Proof of Lemma \ref{lemmacfcm}: Recall, for any $m\in\N\backslash\{1,2\}$, the set  $\U_m=\{U_1,\dots U_{m-1}\}$ where $U_i$ as the open interval $(i,m)\subseteq \R$ gives a open convex realization for $\C_m$ with $X=\R. $ Using this realization in Theorem \ref{thmcurtcan} we get the result of our lemma. 
\begin{proposition}
	Let $m\in\N\backslash\{1,2\}$. Then $G(\C_m)$ is a totally disconnected graph on $m-1$ vertices. 
\end{proposition}
\no The proof of the above proposition comes from the Lemma \ref{lemmacfcm} and the definition of GRG. 

Using Theorem \ref{concomccg}, we see that the GRG and CCG are in duality for complete codes. We end this sub-section with noting down the following theorem. 
\begin{theorem}
Let $\C$ be a complete code. Then the complement of $G(\C)$ is a complete graph.  
\end{theorem}
\subsection{GRG for 2-regular codes}
Recall for any $k\in \N\backslash \{1,2\}$, $ \CR{k}=\{1,2,\dots, k, 12,23,\dots, k-1k,k1\}$ is a 2-regular code with cardinality $2k$. We will study $G(\CR{k})$ in this sub-section.  Using Sage math algorithm given by Petersen et al. \cite{petersen2018neural} we obtain $\CF{\CR{k}}$ for some $k\in N\backslash\{1,2\}$. They are as follows:
\begin{enumerate}
	\item  $\CF{\CR{3}}= \left\{(1-x_{3}) \cdot (1-x_{2}) \cdot (1-x_{1}), x_{3} \cdot x_{2} \cdot x_{1}\right\}$
	\item $\CF{\CR{4}}= \left\{(1-x_{4})\cdot(1-x_{3}) \cdot (1-x_{2}) \cdot (1-x_{1}), x_{3} \cdot x_{1}, x_{4} \cdot x_{2}\right\}$
	\item $\CF{\CR{5}}= \left\{(1-x_{5})\cdot(1-x_{4})\cdot(1-x_{3}) \cdot (1-x_{2}) \cdot (1-x_{1}),x_{3} \cdot x_{1}, x_{4} \cdot x_{1}, x_{4} \cdot x_{2}, x_{5}\cdot x_{2}, x_{5} \cdot x_{3}\right\}$
	\item $\CF{\CR{6}}= \{(1-x_{6})\cdot(1-x_{5})\cdot(1-x_{4})\cdot(1-x_{3}) \cdot (1-x_{2}) \cdot (1-x_{1}),x_{3} \cdot x_{1}, x_{4} \cdot x_{1}, x_{4} \cdot x_{2}, x_{6} \cdot x_{2}, x_{6} \cdot x_{3}, x_{6} \cdot x_{4}, x_{5} \cdot x_{1}, x_{5} \cdot x_{2}, x_{5} \cdot x_{3}\}$
\end{enumerate}
 Figure \ref{Fig2reggrg} represents GRG of the above discussed 2-regular codes. 
\begin{figure}[h!]
	\centering
	\begin{subfigure}[b]{0.2\linewidth}
		\centering 
		\begin{tikzpicture}  
			\vertex (0) at (0,0) [label=below:$3$] {};
			\vertex (1) at (1,0) [label=below:1]{};
			\vertex (12) at (0,1) [label=left:2]{};
			\path[-]
			
			(0) edge (1)
			(0) edge (12)
			
			(1) edge (12)
			;
		\end{tikzpicture}
		\caption{$G({\CR{3}})$} \label{}
	\end{subfigure}
	\begin{subfigure}[b]{0.2\linewidth}
		\centering
		\begin{tikzpicture}  
			\vertex (1) at (0,0) [label=below:1] {};
			\vertex (123) at (1,0) [label=below:2]{};
			\vertex (emp) at (0,1) [label=left:$4$]{};
			\vertex (1234) at (1,1) [label=right:3]{};
			\path[-]
			(1) edge (emp)
			(1) edge (123)
			
			(123) edge (1234)
			
			(emp) edge (1234)
			;
		\end{tikzpicture}
		\caption{$G(\CR{4})$} \label{}
	\end{subfigure}
	\begin{subfigure}[b]{0.2\linewidth}
		\centering
		\begin{tikzpicture}  
			\vertex (1) at (0,0) [label=below:1] {};
			\vertex (2) at (1,0) [label=below:2]{};
			\vertex (4) at (0,1) [label=above:$4$]{};
			\vertex (3) at (1,1) [label=right:3]{};
			\vertex (5) at (-0.707,0.5) [label=left:5]{};
			\path[-]
			(1) edge (5)
			(1) edge (2)
			
			(4) edge (5)
			
			(4) edge (3)
			(2) edge (3)
			;
		\end{tikzpicture}
		\caption{$G(\CR{5})$} \label{}
	\end{subfigure}
	\begin{subfigure}[b]{0.2\linewidth}
		\centering
		\begin{tikzpicture}  
			\vertex (1) at (0,0) [label=below:1] {};
			\vertex (2) at (1,0) [label=below:2]{};
			\vertex (4) at (1,1) [label=above:$4$]{};
			\vertex (3) at (1.5,0.5) [label=right:3]{};
			\vertex (6) at (-0.5,0.5) [label=left:6]{};
			\vertex (5) at (0,1) [label=above:5]{};
			\path[-]
			(1) edge (6)
			(6) edge (5)
			(1) edge (2)
			
			(4) edge (5)
			
			(4) edge (3)
			(2) edge (3)
			;
		\end{tikzpicture}
		\caption{$G(\CR{6})$} \label{}
	\end{subfigure}
	\caption{GRG for some 2-Regular codes}
	\label{Fig2reggrg}
\end{figure}
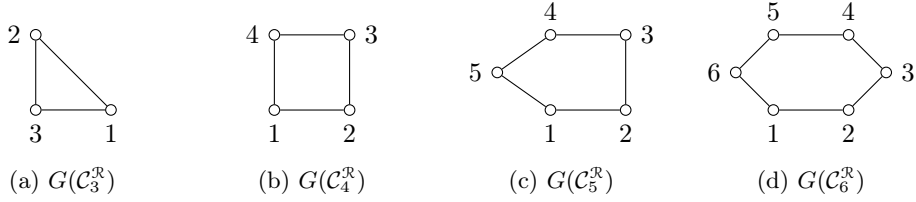
Observe that for $ k\in \{3,4,5,6\}$ we have  $G(\CR{k})$ as a 2 regular graph on $k$ vertices.  We further claim this is true for all $k\in \N\backslash\{1,2\}.$

\begin{lemma} \label{lemmacfreg}
	Let $ \CR{k}=\{1,2,\dots, k, 12,23,\dots, k-1k,k1\}$ for any $ k\in \N \backslash \{1,2,3\}$. Then $$\CF{\CR{k}}=\left\{\displaystyle\prod_{i=1}^{k}(1-x_i) \right\} \cup \{x_ix_j\mid i>j \text{ and } (i-j)\cong 1 \mod k\}.$$
\end{lemma}
\begin{proof}
Let $\U=\{U_1,\dots,U_k\}$ be a realization of $\CR{k}$ as given in Figure \ref{figrealcrk}.	Since, $\emptyset\notin\CR{k}$ we have $X=\bigcup_{i\in[k]} U_i$ we have $\displaystyle\prod_{i=1}^{k}(1-x_i) \in \CF{\CR{k}}$. Further it is clear from the realization that $U_i\cap U_j= \emptyset $ if and only if $i$ and $j$ are not consecutive (Note that we consider 1 and $k$ as consecutive).  Therefore we get the second part of $\CF{\CR{k}}$ using Theorem \ref{thmcurtcan}. Hence the result. 
\end{proof}
\begin{proposition}
	Let $m\in\N\backslash\{1,2\}$. Then $G(\CR{k})$ is a 2-regular graph on $k$ vertices. 
\end{proposition}
\no The proof of the above proposition comes from the Lemma \ref{lemmacfreg} and the definition of GRG. 